\documentclass[12pt, reqno]{amsart}

\title[]{Fujita exponent on stratified Lie groups}

\emergencystretch=\maxdimen
\author[D. Suragan and B. Talwar]{Durvudkhan Suragan and Bharat Talwar}
\address{Durvudkhan Suragan, Department of Mathematics, Nazarbayev University, Astana 010000, Kazakhstan}
\email{durvudkhan.suragan@nu.edu.kz}
\address{Bharat Talwar, Department of Mathematics, Nazarbayev University, Astana 010000, Kazakhstan}
\email{bharat.talwar@nu.edu.kz; btalwar.math@gmail.com}

\allowdisplaybreaks[1]

\usepackage{relsize}
\usepackage{amsmath,amsfonts,amssymb,amsthm,mathrsfs}
\usepackage{hyperref}
\usepackage{cleveref}
\usepackage[margin=2.3cm]{geometry}
\usepackage{setspace}
\usepackage{mathtools}

\thanks{Bharat Talwar was supported by the Collaborative Research Program (project OPCRP2020012) at Nazarbayev University.
This work was partially supported by a grant from the Simons Foundation.}

\usepackage{setspace}
\usepackage{xcolor}

\allowdisplaybreaks[1]

\numberwithin{equation}{section}
\newtheorem{theorem}{\bf Theorem}[section]
\newtheorem{lemma}[theorem]{\bf Lemma}

\newtheorem{remark}[theorem]{\bf Remark}
\newtheorem{prop}[theorem]{\bf Proposition}
\newtheorem{defin}[theorem]{\bf Definition}

\newcommand{\seq}{\subseteq}
\newcommand{\G}{\mathbb{G}}

\newcommand{\N}{\mathbb{N}}

\newcommand{\R}{\mathbb{R}}

\newcommand{\norm}[1]{\| #1 \|}

\makeatletter \@namedef{subjclassname@2020}{\textup{2020} Mathematics Subject Classification} \makeatother
\subjclass[2020]{35R03, 35B33, 35B44}

\begin{document}
\begin{abstract}
We prove that $\frac{Q}{Q-2}$ is the Fujita exponent for a semilinear heat equation on an arbitrary stratified Lie group with homogeneous dimension $Q$.
This covers the Euclidean case and gives new insight into proof techniques on nilpotent Lie groups.
The equation we study has a forcing term  which depends only upon a group element and has positive integral.
The stratified Lie group structure plays an important role in our proofs, along with test function method and Banach fixed point theorem.
\end{abstract}
\keywords{Stratified Lie group, global solution, dilations, subelliptic operator, Fujita exponent}
\maketitle

\section{Introduction}

Let $\G = (\R^n, \star)$ be a stratified Lie group and $\{X_1, \dots, X_m\}$ be a system of generators for $\G$, i.e. a basis for the first strata of $\G$.
For the self-adjoint subelliptic operator $$\Delta_\G = - \sum_{i=1}^m X_i^2,$$ 
and $p > 1$, consider the semilinear subelliptic heat equation
\begin{equation}\label{MainProblem}
\begin{cases} 
u_t(t,x) - \Delta_\G u(t,x) = |u(t,x)|^p + f(x), & (t,x) \in (0,T) \times \G, \\
u(0, x) = u_0(x), & x \in \G,
\end{cases}
\end{equation}
with a non-zero forcing term $f$ depending only upon the space variable $x \in \G$ where the given functions $f, u_0 \in L^1_{loc}(\G)$ are non-negative.
A critical exponent is the largest real number $p_c$ such that (\ref{MainProblem}) has no global solution whenever $p < p_c$.
In the supercritical case of $p > p_c$, there always exists some $f$ and $u_0$ for which (\ref{MainProblem}) has a global solution.

The importance of studying PDEs on stratified Lie groups is apparent from the lifting theorem given by Rothschild and Stein in their celebrated paper \cite{RothschildStein}.
Even though a generic H\"ormander operator which is sum of squares of vector fields does not become a sub-Laplacian, it can be, roughly speaking, approximated by a sub-Laplacian on some stratified Lie group.
This makes it crucial to study the second order partial differential equations on stratified Lie groups.

A vast literature is available on problems similar to (\ref{MainProblem}) and this pursuit started from the seminal paper \cite{Fujita} of Fujita.
For $\G = (\R^n, +)$ and $f \equiv 0$, it was proved that $p_c = \frac{n+2}{n}$.
The critical case $p= p_c$ was then dealt with in \cite{Hayakawa} where it was proved that a global solution does not exist.
In \cite{Bandle} the problem (\ref{MainProblem}) was studied for $\G = (\R^n, +)$ and the authors proved that  $p_c = \frac{n}{n-2}$ when $\int_{\R^n} f(x) dx > 0$.
Then $p_c = \frac{Q+2}{Q}$ was documented by Pascucci in \cite{Pascucci} which provided a natural generalization of Fujita's work on stratified Lie groups $\G = (\R^n, \star)$.
Thus, an essential question was to develop Pascucci's result for the inhomogeneous case.

The main goal of the present paper is to seek an answer to this question.
In this paper, thanks to the available Lie group dilation structure, we extend some results proved in \cite{Bandle} from the abelian case $\G = (\R^n, +)$ to an arbitrary stratified Lie group $\G = (\R^n, \star)$.
While attending to the supercritical case we could not use the classical technique of finding a solution of the form $\epsilon (1 + d(x))^{k}$ ($d$ being a homogeneous norm) since $|\nabla_\G d|$ may not be identically one on $\G \setminus \{0\}$.
This is where the techniques from \cite{Zhang} turned out to be useful.
Employment of the test function method is inspired from \cite{Torebek, JleliKawakamiSamet, GeorgievPalmieri, MitidieriPokhozhaev} and we refer the reader to \cite{RuzhanskyYessirkegenov, Pohozaev, Pascucci, Levine} for further discussions on similar problems.

This paper is organized as follows.
In \Cref{Preliminary} we collect some definitions and results to provide the necessary background and aid smooth reading of the paper.
The main results are proved in \Cref{Results}.
We show in \Cref{subcritical} and \Cref{Critical} that the nonexistence results derived for the Heisenberg group $\mathbb{H}^n$ in \cite[Section 4]{Torebek} can be generalized to $\G$.
In Section 2 of our recent article \cite{SuraganTalwar}, we have already proved that all the results presented in \cite[Section 3]{Torebek} are valid in a more generalized setting, which includes the case of stratified Lie groups.
An existence result is proved in \Cref{Supercritical} which helps us conclude that for (\ref{MainProblem})  $p_c = \frac{Q}{Q-2}$ when $Q > 2$ and $p_c = \infty$ when $Q \in \{ 1,2 \}$.

\section{Preliminaries}\label{Preliminary}
We refer to \cite[Chapter 1]{Bonfiglioli} for the fundamentals of the theory of stratified Lie group $\G = (\R^n, \star)$.
There are $n$ positive integers $1 \leq \sigma_1 \leq  \dots \leq \sigma_n$ associated to $\G$ such that for every $\lambda > 0$ the maps $$\delta_\lambda(x_1, \dots x_n) = (\lambda^{\sigma_1} x_1, \dots \lambda^{\sigma_n} x_n)$$ are group automorphisms known as the dilations on $\G$.
There exists $m \leq n$ linearly independent vector fields $X_1, \dots, X_m$ such that $\{ X_k : 1 \leq k \leq m  \}$ along with its $r \in \N$ commutators generate the Lie algebra of $\G$.
In this case $\sigma_1 =\sigma_2 = \dots \sigma_m = 1$ and we say that $\G$ has step $r$.
The integer $Q = \sum_{i=1}^n  \sigma_i$ is known as the homogeneous dimension of $\G$ and $Q = n$ if and only if  $\G$ is the same as  $(\R^n, +)$.

For $l \in \R$, a vector field $X$ on $\G$ is said to be $\delta_\lambda$-homogeneous of degree $l$ if for every smooth function $f$ on $\G$ the equation $$X(f(\delta_\lambda(x)))= \lambda^l (X f) (\delta_\lambda(x))$$ is satisfied for every $x \in \G$.
Likewise, a real valued function $\phi$ on $\G$ is said to be $\delta_\lambda$-homogeneous of degree $l \in \R$ if $$\phi(\delta_\lambda(x)) = \lambda^l \phi(x)$$ for every $x \in \G$.
As one would expect, if $X$ is $\delta_\lambda$-homogeneous of degree $l_1$ and $\phi$ is $\delta_\lambda$-homogeneous of degree $l_2$, then $X \phi$ is $\delta_\lambda$-homogeneous of degree $l_2 - l_1$ and $X^k$ is $\delta_\lambda$-homogeneous of degree $kl_1$ for every $k \in \N$.

The Carnot-Carath\'eodory (or control) distance as well as the $\Delta_\G$-gauge prove to be of great use in our computations.
Recall that all the homogeneous norms on $\G$ are equivalent.
Till \Cref{Critical} is proved, we use $d$ to denote the Carnot-Carath\'eodory distance.
It gives rise to a homogeneous norm $d(x) := d(x,0)$ on $\G$ satisfying the triangle inequality \cite[p. 4]{Pascucci}.
Then $d$ (and every other homogeneous norm on $\G$) and the smooth vector fields  $X_k$ are $\delta_\lambda$-homogeneous of degree 1 for $1 \leq k \leq m$.
It is known (see e.g. \cite{VSC}) that for the heat kernel $h_t(x,y)$, i.e. the fundamental solution of $\frac{\partial}{\partial t} - \Delta_\G$, the global estimate  \begin{equation}\label{BoundsForHeatKernel}
\frac{1}{C t^{\frac{Q}{2}}} \exp^{\frac{-C d(x,y)^2}{t}} \leq h_t(x,y) \leq  \frac{C}{t^{\frac{Q}{2}}} \exp^{\frac{-d(x,y)^2}{C t}}\end{equation} holds for every $t \in (0, \infty)$.

After proving \Cref{Critical} we use $d$ for $\Delta_\G$-gauge which is by definition a smooth out of origin symmetric  homogeneous  norm satisfying the equation $$\Delta_\G(d^{2-Q}) = 0$$ in $\G \setminus \{0\}$.
Then there exists $c > 0$ such that  $d$ satisfies the pseudo-triangle inequality $$d(x,y) \leq c(d(x) + d(y))$$ for every $x,y \in \G$.
From \cite{Folland} it is  known that for the $\Delta_\G$-gauge  $d$ there exists $\beta_d >0$ such that $$ H(x,y):= \frac{\beta_d}{d(x,y)^{Q-2}}$$ is the fundamental solution for the sub-Laplacian $\Delta_\G$.
One should not confuse $\delta_\lambda$ with $\delta$ which is used as a positive real number at times.
Recall from \cite{Bonfiglioli} that  if $B(x,t)$ denotes the ball in $\G$ of radius $a$ and center at $x$ with respect to $d$, then $|B(x,a)| = a^Q |B(x,1)|$ ($|E|$ being the Haar measure of the subset $E \seq \G$) and for every measurable function $f$ we have (see \cite[Proposition 5.4.4]{Bonfiglioli}) \begin{equation}\label{IntegrationOfRadialFunctions}
\int_{B(0,a)} f(d(x)) dx = Q |B(0,1)| \int_0^a f(\delta) \delta^{Q-1} d\delta.
\end{equation}

Let us fix some notations.
We denote the space of all continuous functions on $\G$ which vanish at infinity by $C_0(\G)$.
For a Banach space $(X, \norm{\cdot})$, the space of all $X$-valued continuous and vanishing at infinity functions on a locally compact Hausdorff topological space $S$ is denoted by $C_0(S,X)$.
This space is a Banach space when the norm of $w \in  C_0(S,X)$ is given by $\norm{w}_{\sup} = \sup \{ \norm{w(s)} : s \in S \}$.
We set $x = (x_1, x_2, \dots, x_n) \in \G$ and similarly $x' = (x_1', x_2', \dots, x_n')$.
Although we use $C$ to denote a constant whose value may keep on changing from line to line $2C$ is also used at some places to make the inequalities easier to understand.

\section{Fujita exponent is $\frac{Q}{Q-2}$}\label{Results}

Let us fix $\G^{\ast} = \G \setminus \{0\}$ and $\frac{1}{p} + \frac{1}{p'} = 1$.

\begin{defin}\label{DefinitionLocalWeakSolution}
	Let $u_0 \in L^1_{loc}(\G)$.
	Then $u \in L^p_{loc}((0,T), L^p_{loc}(\G))$ is called a local in time weak solution of (\ref{MainProblem}) if for every non-negative test function $\psi \in C^1_0( (0,T); C_0^2(\G^{\ast}) )$ satisfying $\psi(T, \cdot) = 0$ we have
\begin{eqnarray}
	\int_0^T \int_{\G^{\ast}} |u|^p \psi + \int_{\G^{\ast}} u_0 \psi(0, x) + \int_0^T \int_{\G^{\ast}} f \psi + \int_0^T \int_{\G^{\ast}} u \psi_t + \int_0^T \int_{\G^{\ast}} u \Delta_\G \psi = 0.
\end{eqnarray}
When $T = \infty$, such $u$ is called a global in time weak solution.
\end{defin}

Let us start with providing a result  of independent interest in relation to which we refer the reader to \cite[Theorem 2.5]{MontiSerra} where a similar result is proved in a different setting.
It is the idea of the proof of \Cref{NabladIsBounded} which helps us to obtain the non-existence results later in the paper.
\begin{prop}\label{NabladIsBounded}
For the control distance $d$ on $\G$, the function $|\nabla_\G d|$ is bounded on $\G^{\ast}$.
\end{prop}
\begin{proof}
Let us put to use the fact that the vector fields $\{ X_k: 1 \leq k \leq m \}$ and $d$ are  $\delta_\lambda$-homogeneous of degree $1$.
This tells us that $X_k d$ is $\delta_\lambda$-homogeneous of degree $0$ which means that for every $x \in \G$ and $\lambda > 0$ $$(X_kd) (\delta_{\lambda}(x)) = (X_kd)(x).$$
Thus, $$\sup_{x \in \G^{\ast}} \{ (X_kd)(x) \} = \sup_{x \in \G : d(x) = 1} \{ (X_kd)(x) \}.$$
As $X_kd$ is continuous on the compact sphere $\{x \in \G : d(x) = 1\}$, we obtain that  this supremum is finite for every $1 \leq k \leq m$.
Hence, $$|\nabla_\G d|^2 = \sum_{k=1}^m (X_k d)^2$$ is bounded on $\G^{\ast}$.	
\end{proof}

Let us now discuss the subcritical case.
\begin{theorem}\label{subcritical}
Let $Q \geq 3$, $p < \frac{Q}{Q-2}$, $u_0 \geq 0$ and $\int_{\G^\ast} f(x) dx > 0$.
Then (\ref{MainProblem}) does not admit a global in time weak solution.
\end{theorem}
\begin{proof}
We adapt the proof of \cite[Theorem 3.2]{SuraganTalwar} and use some delicate reasoning.
Let us suppose on the contrary that $u$ is a global in time weak solution.
Then for every test function $\psi$, $$\int_0^T \int_{\G^{\ast}} |u|^p \psi + \int_{\G^{\ast}} u_0 \psi(0, x) + \int_0^T \int_{\G^{\ast}} f \psi + \int_0^T \int_{\G^{\ast}} u \psi_t + \int_0^T \int_{\G^{\ast}} u \Delta_\G \psi = 0.$$
As both $\psi$ and $u_0$ are non-negative, we obtain
\begin{eqnarray}\label{Inequality}
\int_{0}^T \int_{\G^{\ast}} f \psi + \int_{0}^T \int_{\G^{\ast}} |u|^p \psi &\leq& - \int_{0}^T \int_{\G^{\ast}} u \psi_t   - \int_{0}^T \int_{\G^{\ast}} u \Delta_\G \psi \nonumber  \\
& \leq &  \int_{0}^T \int_{\R^n} | u || \psi_t| + |u|| \Delta_\G \psi|.
\end{eqnarray}
Set $P(x) = \frac{d(x)}{\sqrt{T}} $.
Then $P(x) = (p \circ d)(x)$ where  $p: \R \to \R$ is given by $$p(w)= \frac{w}{\sqrt{T}}.$$
Define a test function (see \Cref{DefinitionLocalWeakSolution}) with separated variables as $$\psi(t,x) = \Phi(P(x)) \Phi(2t/T),$$ where \begin{equation*}
\Phi(z) = \begin{cases} 
1, & z \in [0,1],\\
\searrow, & z \in [1, 2],\\
0, & z \in (2, \infty),
\end{cases}
\end{equation*}
is a function in $C^2_0(\R^+)$ satisfying, for $$(\G^{\ast}) := \delta_{\frac{1}{\sqrt{T}}} \G^{\ast} := \left\{\delta_{\frac{1}{\sqrt{T}}}(x) : x \in \G^{\ast} \right\},$$ the estimates $$\int_{(\G^{\ast})} \Phi(P(x'))^{\frac{-1}{p-1}} dx' < \infty$$ and $$\int_{0}^2    |\Phi(t')|^{\frac{-1}{p-1}}|  \Phi'(t')|^{\frac{p}{p-1}} dt' < \infty$$  where $t' = \frac{2t}{T}$ and $x' = \delta_{\frac{1}{\sqrt{T}}} x$.

Note that $$\text{supp}(\psi) \seq \left\{ (t,x) \in \R^+ \times \G^{\ast} : 0 \leq 2t/T \leq 2, 0 \leq P(x) \leq 2 \right\}. \smallskip$$
Now $\psi_t(x,t) = \frac{\partial}{\partial t}(\Phi(P(x)) \Phi(2t/T)) = \frac{2}{T}  \Phi(P(x)) \Phi'(2t/T)$ implies that $$\text{supp}(\psi_t) \seq \left\{ (t,x) \in \R^+ \times \G^{\ast} : 1 \leq 2t/T \leq 2, 0 \leq P(x) \leq 2 \right\}.$$
To keep the computations for $\Delta_\G \psi(t,x) = \Phi(2t/T) \Delta_\G \Phi(P(x))$  clear, we define $\Psi = \Phi \circ p$.
Then $\Psi$ is a smooth function from $\R$ to $\R$.
Note that
$$\Delta_\G \psi(t,x) = \Phi(2t/T) \Delta_\G \Phi(P(x)) = \Phi(2t/T) \left(\Delta_\G  (\Psi \circ d)\right)(x),$$
$$\Psi' = (\Phi \circ p)' =  \Phi'(p) p'$$ and $$\Psi'' = (\Phi \circ p)'' = (\Phi'(p) p')' = (\Phi'(p))' p' =  \Phi''(p) (p')^2.$$
Then
\begin{eqnarray*}
	\left(\Delta_\G  (\Psi \circ d)\right)(x) &=&   \sum_{k=1}^m  \left(X_k^2 (\Psi \circ d)\right) (x)\\
	&=&  \sum_{k=1}^m \left( \Psi^{''}(d(x)) (X_k(d(x)))^2 + \Psi'(d(x)) X_k^2(d(x)) \right)\\
	&=&  \sum_{k=1}^m \left( \Phi^{''}(P(x)) \frac{1}{T} \left( X_k(d(x)) \right)^2
	+ \frac{\Phi^{'}(P(x))}{\sqrt{T}} \left( X_k^2(d(x))  \right) \right).
\end{eqnarray*}
implies that \[\text{supp}(\Delta_\G \psi) \seq \left\{ (t,x) \in \R^+ \times \G^{\ast} : 0 \leq 2t/T \leq 2, 1 \leq P(x) \leq 2 \right\}.\]

For every $1 \leq k \leq m$, as $d$ and $X_k$ are $\delta_\lambda$-homogeneous of degree $1$ and $X_k^2$ is $\delta_\lambda$-homogeneous of degree $2$, we obtain \begin{eqnarray*}
\left(\Delta_\G  (\Psi \circ d)\right)(\delta_{\sqrt{T}} (x))
&=& \sum_{k=1}^m \left( \Phi^{''}(P(\delta_{\sqrt{T}} (x))) \frac{1}{T} \left( X_k(d(\delta_{\sqrt{T}} (x))) \right)^2 \right.  \\
&+& \left. \frac{\Phi^{'}(P(\delta_{\sqrt{T}} (x)))}{\sqrt{T}} \left( X_k^2(d(\delta_{\sqrt{T}} (x)))  \right) \right)\\
	&=& \sum_{k=1}^m \left( \Phi^{''}(P(\delta_{\sqrt{T}} (x))) \frac{1}{T} \left( X_k(d(x)) \right)^2
	+ \frac{\Phi^{'}(P(\delta_{\sqrt{T}} (x)))}{T} \left( X_k^2(d(x)) \right) \right).
\end{eqnarray*}
Consider the compact sphere  $$S = \left\{ x \in \G : d(x) = 1 \right\}.$$
Every element of $\G$ can be written as $\delta_{\frac{1}{\sqrt{T}}} (x)$ for some $x \in S$ and $T > 0$.
Also, since the vector fields under discussion are smooth $X_k^2(d(x))$ and $X_k(d(x))$ must be bounded for $x \in S$.
Finally, as $\Phi$ is smooth and compactly supported, there exists $C > 0$ such that $$ |\Delta_\G  (\Psi \circ d)| \leq \frac{C}{T}.$$

Using inequality 3.2, $\frac{p}{2}$-Young inequality and the similar arguments as in \cite[Theorem 3.2]{SuraganTalwar} one obtains
\[\int_{0}^T \int_{\G^{\ast}} f \psi \leq \frac{1}{p' (\frac{p}{2})^{p'-1}} \Big( \int_{0}^T \int_{\G^{\ast}} |\psi|^{\frac{-1}{p-1}}|\Delta_\G \psi|^{\frac{p}{p-1}} + |\psi|^{\frac{-1}{p-1}}|\psi_t|^{\frac{p}{p-1}} \Big).\]

Now, \begin{eqnarray*}
\int_{0}^T \int_{\G^{\ast}} |\psi|^{ \frac{-1}{p-1}}|\Delta_\G \psi|^{\frac{p}{p-1}} & \leq  & \int_{0}^T \int_{\G^{\ast}} |\Phi(P(x)) \Phi(2t/T)|^{\frac{-1}{p-1}} \left(\frac{C}{T}\right)^{\frac{p}{p-1}} \\
& =& \left(\frac{C}{T}\right)^{\frac{p}{p-1}} \int_{0}^T \int_{\G^{\ast}} |\Phi(P(x)) \Phi(2t/T)|^{\frac{-1}{p-1}} \\
&= & \left(\frac{C}{T}\right)^{\frac{p}{p-1}} T^{\frac{Q}{2}} \frac{T}{2} \int_{0}^2 \int_{(\G^{\ast})} |\Phi(P(x')) \Phi(t')|^{\frac{-1}{p-1}} \\
&= & \left(\frac{C}{T}\right)^{\frac{p}{p-1}} T^{\frac{Q}{2}} T \int_{(\G^{\ast})} |\Phi(P(x'))|^{-1/(p-1)} \leq C_1 T^{\frac{Q}{2}+1 - \frac{p}{p-1}}
\end{eqnarray*}
and similarly,
\begin{eqnarray*}
\int_{0}^T \int_{\G^{\ast}}  |\psi|^{\frac{-1}{p-1}}|\psi_t|^{\frac{p}{p-1}} &=&  \int_{0}^T \int_{\G^{\ast}}  |\Phi(P(x)) \Phi(2t/T)|^{\frac{-1}{p-1}} \left|\frac{2}{T}  \Phi(P(x)) \Phi'(2t/T)\right|^{\frac{p}{p-1}} \\
& = &  \left(\frac{2}{T}\right)^{{\frac{p}{p-1}}} \int_{0}^T \int_{\G^{\ast}}  |\Phi(P(x))|  |\Phi(2t/T)|^{\frac{-1}{p-1}}|  \Phi'(2t/T)|^{\frac{p}{p-1}}\\
& = & \left(\frac{2}{T}\right)^{{\frac{p}{p-1}}} \frac{T}{2} T^{\frac{Q}{2}} \int_{0}^2 \int_{(\G^{\ast})}  |\Phi(P(x'))|  |\Phi(t')|^{\frac{-1}{p-1}}|  \Phi'(t')|^{\frac{p}{p-1}}.
\end{eqnarray*}

As, \begin{equation*}
\int_{0}^T \int_{\G^{\ast}} f \psi = \int_{0}^T \Phi(2t/T) \int_{\G^{\ast}} f \Phi(P(x))  =  \frac{T}{2} \int_{0}^2 \Phi(t') \int_{\G^{\ast}} f \Phi(P(x)) \geq T \int_{\G^{\ast}} f \Phi(P(x))
\end{equation*}
we obtain $$T \int_{\G^{\ast}} f \Phi(P(x)) \leq \int_{0}^T \int_{\G^{\ast}} f \psi \leq C_1 T^{\frac{Q}{2}+1 - \frac{p}{p-1}}.$$
By dominated convergence theorem, $ \int_{\G^{\ast}} f \Phi \to \int_{\G^{\ast}} f $ as $T \to \infty$.
Using the fact that $p < \frac{Q}{Q-2}$ implies $\frac{Q}{2} < \frac{p}{p-1}$ we obtain $\int_{\G^{\ast}} f \leq 0$ by taking $T \to \infty$.
This contradicts the hypothesis that $\int_{\G^{\ast}} f > 0$.
So, no such $u$ exists and the statement of the theorem is proved.
\end{proof}

\begin{remark}
Note from the proof of \Cref{subcritical} that the assumption $p < \frac{Q}{Q-2}$ was used near the end.
Thus, we have the inequality, $$T \int_{\G^{\ast}} f \Phi(P(x)) \leq \int_{0}^T \int_{\G^{\ast}} f \psi \leq C_1 T^{\frac{Q}{2}+1 - \frac{p}{p-1}}$$ for every $p > 1$ and $Q \in \N$.
Since for $Q \in  \{ 1,2\}$ and $p > 1$ we have $\frac{Q}{2} - \frac{p}{p-1} < 0$, one can take $T \to \infty$ and conclude that $\int_{\G^{\ast}} f \leq 0$.
Hence, $p_c = \infty$ when $Q \in  \{ 1,2\}$, $u_0 \geq 0$ and $\int_{\G^\ast} f(x) dx > 0$.
\end{remark}

We now discuss the critical case.
\begin{theorem}\label{Critical}
Let $p = \frac{Q}{Q-2}$, $u_0 \geq 0$ and $\int_{\G^\ast} f(x) dx > 0$.
Then (\ref{MainProblem}) does not admit a global in time weak solution.
\end{theorem}
\begin{proof}
Let us suppose on the contrary that $u$ is a global in time weak solution for (\ref{MainProblem}).
For any $0 < R < \infty$, set $P(x) := \frac{ \ln\left(\frac{d(x)}{\sqrt{R}}\right) }{\ln(\sqrt{R})}$ and define a test function (see \Cref{DefinitionLocalWeakSolution}) with separated variables as $\psi(t,x) = \Phi(P(x)) \Phi(t/T)$ where \begin{equation*}
\Phi(z) = \begin{cases} 
1, & z \in (- \infty ,0],\\
	\searrow, & z \in [0, 1],\\
	0, & z \in (1, \infty),
	\end{cases}
\end{equation*}
is a function in $C^\infty(\R)$ satisfying, for $$(\G^{\ast}) := \delta_{\frac{1}{\sqrt{R}}} \G^{\ast} := \left\{\delta_{\frac{1}{\sqrt{R}}}(x) : x \in \G^{\ast} \right\},$$ the estimates $$\int_{(\G^{\ast})} \Phi(P(x'))^{\frac{-1}{p-1}} dx' < \infty$$ and $$\int_{0}^2    |\Phi(t')|^{\frac{-1}{p-1}}|  \Phi'(t')|^{\frac{p}{p-1}} dt' < \infty$$  where $t' = \frac{t}{T}$ and $x' = \delta_{\frac{1}{\sqrt{R}}} (x)$.
Note that $$\text{supp}(\psi) \seq \left\{ (t,x) \in \R^+ \times \G^{\ast} : 0 \leq t \leq T, - \infty < P(x) \leq 1 \right\}.$$
As $- \infty \leq P(x) \leq 1$ if and only if $d(x) \leq R$, the set $\text{supp}(\psi)$ is compact.
The fact that $$\psi_t(x,t) = \frac{\partial}{\partial t}(\Phi(P) \Phi(t/T)) = \frac{1}{T}  \Phi(P) \Phi'(t/T)$$ implies the embedding   $$\text{supp}(\psi_t) \seq \left\{ (t,x) \in \R^+ \times \G^{\ast} : 0 \leq t \leq T, - \infty < P(x) \leq 1 \right\}.$$

To keep the computations for $\Delta_\G \psi(t,x) = \Phi(2t/T) \Delta_\G \Phi(P(x))$ clear, let $$p(w) = \frac{\ln(w)}{\ln({\sqrt{R}})}.$$
Then $p'(w) = \frac{1}{w \ln({\sqrt{R}})}$ and $p''(w) = -\frac{1}{w^2 \ln({\sqrt{R}})}$ for every $w \in \R^+$.
For $\Psi = \Phi \circ p$, which is a smooth function from $\R^{+}$ to $\R$, we have
$$\Psi' = (\Phi \circ p)' =  \Phi'(p) p'$$ and $$\Psi'' = (\Phi \circ p)'' = (\Phi'(p) p')' = (\Phi'(p))' p' + \Phi'(p) p''  =  \Phi''(p) (p')^2 + \Phi'(p) p''$$ which implies that 
\begin{eqnarray*}
\left(\Delta_\G  (\Psi \circ d)\right)(x) &=&   \sum_{k=1}^m  \left(X_k^2 (\Psi \circ d)\right) (x)\\
&=&  \sum_{k=1}^m \left( \Psi^{''}(d(x)) (X_k(d(x)))^2 + \Psi'(d(x)) X_k^2(d(x)) \right)\\
&=&  \sum_{k=1}^m \left( \left( \Phi^{''}(P(x)) \frac{1}{d(x)^2 \ln(\sqrt{R})^2} - \frac{\Phi^{'}(P(x))}{d(x)^2 \ln(\sqrt{R})} \right) \left( X_k(d(x)) \right)^2  \right. \\
& + & \left.  \frac{\Phi^{'}(P(x))}{d(x) \ln(\sqrt{R})}   \left( X_k^2(d(x))  \right)  \right).
\end{eqnarray*}

Hence $$\text{supp}(\Delta_\G(\psi)) \seq \left\{ (t,x) : 0 \leq P \leq 1, 0 \leq t \leq T \right\}.\smallskip$$
For large $R$  we have $\ln(\sqrt{R}) \geq 0$.
Thus,
\begin{eqnarray*}
0 \leq P(x) \leq 1 &\iff&   0 \leq  \ln(\frac{d(x)}{\sqrt{R}}) \leq \ln(\sqrt{R})\\
& \iff&  \exp{(0)} \leq  \exp{\left(\ln\left(\frac{d(x)}{\sqrt{R}}\right)\right)} \leq \exp{(\ln(\sqrt{R}))}     \\
&   \iff & 1 \leq \frac{d(x)}{\sqrt{R}} \leq \sqrt{R} \iff \sqrt{R} \leq d(x) \leq R,
\end{eqnarray*}
implying that $$\text{supp}(\Delta_\G(\psi)) \seq \left\{ ( t,x) : \sqrt{R} \leq d(x) \leq R, 0 \leq t \leq T \right\}.$$
For $(t,x) \in \R^+ \times \G^{\ast}$, just as in the subcritical case, we have
 \begin{eqnarray*}
\left(\Delta_\G  (\Psi \circ d)\right)(\delta_{\sqrt{R}} (x))
&=&  \sum_{k=1}^m \left( \frac{\Phi^{'}(P(\delta_{\sqrt{R}} (x)))}{d(\delta_{\sqrt{R}} (x)) \ln(\sqrt{R})}   \left( X_k^2(d(\delta_{\sqrt{R}} (x)))  \right) \right. \\
& -& \left.   \frac{\Phi^{'}(P(\delta_{\sqrt{R}} (x)))}{d(\delta_{\sqrt{R}} (x))^2 \ln(\sqrt{R})} \left( X_k(d(\delta_{\sqrt{R}} (x))) \right)^2    \right. \\
& + & \left.  \frac{ \Phi^{''}(P(\delta_{\sqrt{R}} (x)))}{d(\delta_{\sqrt{R}} (x))^2 \ln(\sqrt{R})^2} \left( X_k(d(\delta_{\sqrt{R}} (x))) \right)^2    \right)\\
&=&  \sum_{k=1}^m \left( \frac{\Phi^{'}(P(\delta_{\sqrt{R}} (x)))}{\sqrt{R} d(x) \ln(\sqrt{R})}   \left( \frac{1}{\sqrt{R}} X_k^2(d(x))  \right) \right. \\
& -& \left.  \frac{\Phi^{'}(P(\delta_{\sqrt{R}} (x)))}{R d(x)^2 \ln(\sqrt{R})} \left( X_k(d(x)) \right)^2    \right. \\
& + & \left.  \frac{ \Phi^{''}(P(\delta_{\sqrt{R}} (x)))}{R d(x)^2 \ln{(\sqrt{R})}^2} \left( X_k(d(x)) \right)^2    \right).
\end{eqnarray*}
Considering the sphere $S$ as in \Cref{subcritical}, we obtain that $$|\Delta_\G  (\Psi \circ d)| \leq  \frac{C}{R \ln{(\sqrt{R})}}$$ and $$\int_{0}^T \int_{\G^{\ast}} f \psi \leq \frac{1}{p' (\frac{p}{2})^{p'-1}} \Big( \int_{0}^T \int_{\G^{\ast}} |\psi|^{\frac{-1}{p-1}}|\Delta_\G \psi|^{\frac{p}{p-1}} + |\psi|^{\frac{-1}{p-1}}|\psi_t|^{\frac{p}{p-1}} \Big).$$

With change of the variables $x \to x'$ and $t \to t'$ we obtain  \begin{eqnarray*}
\int_{0}^T \int_{\G^{\ast}} |\psi|^{\frac{-1}{p-1}}|\Delta_\G \psi|^{\frac{p}{p-1}} & \leq  & \int_{0}^T \int_{\G^{\ast}} |\Phi(P(x)) \Phi(t/T)|^{\frac{-1}{p-1}} \left(\frac{C}{\ln{(\sqrt{R})}  R }\right)^{\frac{p}{p-1}} \\
& \leq & \left(\frac{C}{\ln{(\sqrt{R})}  R }\right)^{\frac{p}{p-1}} \int_{0}^T \int_{\G^{\ast}} |\Phi(P(x)) \Phi(t/T)|^{\frac{-1}{p-1}} \\
&\leq & \left(\frac{C}{\ln{(\sqrt{R})}  R }\right)^{\frac{p}{p-1}} R^{\frac{Q}{2}} T \int_{0}^2 \int_{(\G^{\ast})} |\Phi(P(x')) \Phi(t')|^{\frac{-1}{p-1}} \\
&\leq & \left(\frac{C}{\ln{(\sqrt{R})}  R } \right)^{\frac{Q}{2}} R^{\frac{Q}{2}} 2 T \int_{(\G^{\ast})} |\Phi(P(x'))|^{\frac{-1}{p-1}} \\
&\leq & C_1 T  \ln{(\sqrt{R})^{- \frac{Q}{2}}}
\end{eqnarray*} 
and similarly,
\begin{eqnarray*}
\int_{0}^T \int_{\G^{\ast}}  |\psi|^{\frac{-1}{p-1}}|\psi_t|^{\frac{p}{p-1}} &\leq&  \int_{0}^T \int_{\G^{\ast}}  |\Phi(P(x)) \Phi(t/T)|^{\frac{-1}{p-1}} \left|\frac{1}{T}  \Phi(P(x)) \Phi'(t/T) \right|^{\frac{p}{p-1}} \\
& \leq &  \left(\frac{1}{T}\right)^{\frac{p}{p-1}} \int_{0}^T \int_{\G^{\ast}}  |\Phi(P(x))|  |\Phi(t/T)|^{\frac{-1}{p-1}}|  \Phi'(t/T)|^{\frac{p}{p-1}}\\
& \leq & \left(\frac{1}{T}\right)^{\frac{Q}{2}} T R^{\frac{Q}{2}} \int_{0}^2 \int_{(\G^{\ast})}  |\Phi(P'(x))|  |\Phi(t')|^{\frac{-1}{p-1}}|  \Phi'(t')|^{\frac{p}{p-1}}\\
& \leq & C_1 T^{\frac{-Q}{2}} T R^{\frac{Q}{2}}.
\end{eqnarray*}
	
	Moreover, \begin{eqnarray*}
		\int_{0}^T \int_{\G^{\ast}} f \psi &=& \int_{0}^T \Phi(t/T) \int_{\G^{\ast}} f \Phi(P(x))  =  T \int_{0}^1 \Phi(t') \int_{\G^{\ast}} f \Phi(P(x)) \geq  C_2 T \int_{\G^{\ast}} f \Phi(P(x))
	\end{eqnarray*}
	implies $$\int_{\G^{\ast}} f \Phi(P(x)) \leq C' (  \ln{(\sqrt{R})^{- \frac{Q}{2}}} +  T^{\frac{-Q}{2}} R^{\frac{Q}{2}} ).$$
Putting $T^2 = R$ and then taking $R \to \infty$ we obtain $\int_{\G^{\ast}} f \leq 0$ which  contradicts the hypothesis that $\int_{\G^{\ast}} f > 0$.
So, no such $u$ exists and the statement is proved.
\end{proof}

One major difficulty that we encounter is that for a $\Delta_\G$-gauge $d$, the function $\nabla_\G d$ is identically one on $\G^{\ast}$ if and only if $\G$ is $(\R^n, +)$ \cite[p. 466]{Bonfiglioli}.
This is what makes it difficult to use the classical  technique of finding a function of the form $u(x) = \epsilon (1 + d(x))^{k}$ for suitable $k$ and $\epsilon$ such that $-\Delta u - u^p \geq 0$.
However, we now show that the ideas of \cite{Zhang} work on $\G$ when we take $d$ to be a $\Delta_\G$-gauge satisfying $d(x,y) \leq c (d(x) + d(y))$ for every $x,y \in \G$.
For this we study (\ref{MainProblem}) using an integral equation.
\begin{defin}(Mild solution)
A local in time mild solution of (\ref{MainProblem}) is a function $u \in C([0,T], C_0(\G))$ such that $$u(t) = e^{-t \Delta_\G} u(0) + \int_{0}^t e^{-(t-s) \Delta_\G} (|u(s)|^p + f) ds,$$ for any $ t \in [0,T)$.
The function $u$ is called a global mild solution of (\ref{MainProblem}) if $T = + \infty$.
\end{defin}

It is not difficult to see that the proof of \cite[Lemma 3.4]{Torebek} works on $\G$ and hence every mild solution is a weak solution as well.
To be able to do some analysis about existence of global positive mild solutions, we need a couple of indispensable lemmas.

\begin{lemma}\label{Proposition1OfZhang}
	For $\delta > 0$, there exists a constant $C$ depending only upon $\delta, Q$ and $w \geq 2$ such that
	$$\sup_{x \in \G} \int_\G \frac{1}{d(x,y)^{Q-2}} \frac{1}{1 + d(y)^{w+ \delta}} dy < C.$$
\end{lemma}
\begin{proof}
Since the Haar measure on $\G$ is invariant under translation, the proof is same as that of \cite[Proposition 2.1]{Zhang} except that we need to use \Cref{IntegrationOfRadialFunctions} to obtain
\begin{eqnarray*}
		\int_\G \frac{1}{d(y)^{Q-2}} \frac{1}{1 + d(y)^{w + \delta}} dy &=&  \lim_{s \to \infty} \int_{\{ y \in \G: d(y) \leq s \}} \frac{1}{d(y)^{Q-2}} \frac{1}{1 + d(y)^{w + \delta}} dy \\
		&  = &  C \lim_{s \to \infty} \int_0^s \frac{a^{Q-1}}{(1 + a^{w + \delta})a^{Q-2}} da.
	\end{eqnarray*}
	Since $w \geq 2$ and $\delta> 0$ this integral is finite, giving the result.
\end{proof}

\begin{lemma}\label{Proposition2OfZhang}
For $\delta > 0$ and every $x,y \in \G$ we have
	$$\int_\G \frac{1}{d(x,y)^{Q-2}} \frac{1}{1 + d(y)^{Q+ \delta}} dy \leq \frac{C}{1+ d(x)^{Q-2}}$$ for some constant $C$ which is independent of $x$.
\end{lemma}
\begin{proof}
The proof is refinement of the proof of \cite[Proposition 2.2]{Zhang} and we provide an elaborate reasoning for the sake of completeness.
When $d(x) \leq 1$, with $w = Q$ in \Cref{Proposition1OfZhang} we obtain $$ \int_\G \frac{1}{d(x,y)^{Q-2}} \frac{1}{1 + d(y)^{Q+ \delta}} dy \leq C = \frac{2C}{2} \leq \frac{2 C}{1+ d(x)^{Q-2}}.$$
	
Let us assume that $d(x) > 1$.
If $d(y) \leq \frac{d(x)}{2c}$ then the pseudo-triangle inequality  $d(x) \leq c(d(x,y) + d(y))$ tells us that  $\frac{d(x)}{2c} \leq d(x,y)$.
Moreover, if $d(y) \geq \frac{d(x)}{2c}$, then using the assumption that  $d(x) \geq 1$ we obtain
\begin{eqnarray*}
\frac{1}{1 + d(y)^{Q+ \delta}} &=& \frac{1}{1 + d(y)^{Q - 2} d(y)^{\delta + 2}} \leq  \frac{C}{1 + d(x)^{Q - 2} d(y)^{\delta + 2}} \\
&\leq&  \frac{2 C}{(1 + d(x)^{Q - 2}) (1 + d(y)^{2 + \delta}) }.
\end{eqnarray*}
With $w = 2$ in \Cref{Proposition1OfZhang}, we have
\begin{eqnarray*}
\int_\G \frac{1}{d(x,y)^{Q-2}} \frac{1}{1 + d(y)^{Q+ \delta}} dy&=& \int_{d(y) \geq \frac{d(x)}{2c}} \frac{1}{d(x,y)^{Q-2}} \frac{1}{1 + d(y)^{Q+ \delta}} dy \\
&+& \int_{d(y) \leq \frac{d(x)}{2c}} \frac{1}{d(x,y)^{Q-2}} \frac{1}{1 + d(y)^{Q+ \delta}} dy \\
&\leq& \int_{d(y) \geq \frac{d(x)}{2c}} \frac{C}{(1 + d(x)^{Q - 2}) (1 + d(y)^{2 + \delta})} \frac{1}{d(x,y)^{Q-2}} dy \\
&+& \frac{C}{d(x)^{Q-2}} \quad \text{(From \Cref{IntegrationOfRadialFunctions})} \\
&\leq& \frac{C}{1 + d(x)^{Q-2}}  \int_{d(y) \geq \frac{d(x)}{2c}} \frac{1}{d(x,y)^{Q-2}} \frac{1}{1 + d(y)^{2+ \delta}} dy \\
&+& \frac{C}{d(x)^{Q-2}} \\
		&\leq& \frac{C}{1 + d(x)^{Q-2}}  + \frac{C}{d(x)^{Q-2}} \leq \frac{C}{d(x)^{Q-2}} \\
		& \leq &  \frac{2 C}{1 + d(x)^{Q-2}} \ \ (\text{as } d(x) > 1).
	\end{eqnarray*}
	This completes the proof.
\end{proof}

\begin{theorem}\label{Supercritical}
	For $p > \frac{Q}{Q-2}$, there exists $f$ and $u_0$ such that a positive global mild solution of (\ref{MainProblem}) exists.
\end{theorem}
\begin{proof}
We intend to use the Banach fixed point theorem.
	For $M >0$, let
	$$ S_M = \left\{ u \in C([0, \infty), C_0(\G) ) : 0 \leq u(t)(x) \leq \frac{M}{1 + d(x)^{Q-2}} \text{ for every } t \in [0, \infty) \right\}.$$
Then $S_M$ is a complete metric space and we will be able to apply the Banach fixed point theorem to any contraction from $S_M$ to itself.
Define $$T_M : S_M \to S_M$$ by $$T_M(u)(t) = e^{-t \Delta_\G} u(0) + \int_{0}^t e^{-(t-s) \Delta_\G} (u(s)^p + f) ds.$$
The result will  follow if we can prove that $T_M$ is a well-defined contraction for some $M$.

As $p > \frac{Q}{Q-2}$, there exists $\delta > 0$ such that $\frac{1}{(1 + d(x)^{Q-2})^p} \leq \frac{C}{(1 + d(x)^{Q+ \delta})}$ for every $x \in \G$.
For $\epsilon > 0$,  fix $$u(0)(x) = f(x) = \frac{\epsilon}{1+ d(x)^{Q + \delta}}.$$
To prove that $T_M$ is well-defined, we show that there exists $C > 0$ for which $$e^{-t \Delta_\G} u(0)(x) \leq \frac{C}{1 + d(x)^{Q-2}}$$ and $$\left (\int_{0}^t e^{-(t-s) \Delta_\G} (u(s)^p + f) ds \right )(x) \leq \frac{C}{1 + d(x)^{Q-2}}.$$
Using the pseudo-triangle inequality $d(x) \leq c(d(y) + d(x,y))$ for every $x,y \in \G$, we obtain 
\begin{eqnarray*}
e^{-t \Delta_\G} u(0)(x) &=& \int_\G h_t(x,y) u(0)(y) dy = \epsilon \int_\G  \frac{h_t(x,y)}{1+ d(y)^{Q + \delta}} dy \\
&\leq &  \frac{\epsilon C}{t^{\frac{Q}{2}}}  \int_\G  \frac{1}{1+ d(y)^{Q + \delta}} \exp^{\frac{-d(x,y)^2}{C t}} dy\\
&=& \frac{\epsilon C}{t^{\frac{Q}{2}}}  \left[  \int_{d(y) > \frac{d(x)}{2c}}  \frac{1}{1+ d(y)^{Q + \delta}} \exp^{\frac{-d(x,y)^2}{C t}} dy \right. \\
&+& \left. \int_{d(y) \leq \frac{d(x)}{2c}}  \frac{1}{1+ d(y)^{Q + \delta}} \exp^{\frac{-d(x,y)^2}{C t}} dy \right] \\
&\leq& \frac{\epsilon C}{t^{\frac{Q}{2}}}  \left[  \frac{1}{1+ (\frac{d(x)}{2c})^{Q + \delta}}  \int_{d(y) > \frac{d(x)}{2c}}  \exp^{\frac{-d(x,y)^2}{C t}} dy \right. \\
&+& \left. \int_{d(y) \leq \frac{d(x)}{2c}}  \frac{1}{1+ d(y)^{Q + \delta}} \exp^{\frac{-d(x,y)^2}{C t}} dy \right].
\end{eqnarray*}
Since $\frac{d(x,y)^2}{C t} = d\left(\delta_{\frac{1}{\sqrt{Ct}}}x,\delta_{\frac{1}{\sqrt{Ct}}}y\right)^2$, with change of variable $x' \leftrightarrow \delta_{\frac{1}{\sqrt{Ct}}}x$ we obtain $$\int_\G \exp^{\frac{-d(x,y)^2}{C t}} dy =  \int_\G \sqrt{Ct}^{Q}\exp^{-d(x',y')^2} dy' = C t^{\frac{Q}{2}} \int_\G \exp^{-d(x',y')^2} dy' \leq C t^{\frac{Q}{2}}$$ where $C$ is independent of $x$ and the integral $\int_\G \exp^{-d(x',y')^2} dy'$ is finite because (\ref{IntegrationOfRadialFunctions}) applies.

We separately deal with the two possible cases.
Let us first assume that $d(x) > 1$.
Using the fact that $d(x) \leq c( d(y) + d(x,y))$ one obtains
	\begin{eqnarray*}
e^{-t \Delta_\G} u(0)(x) &\leq& \frac{\epsilon C}{t^{\frac{Q}{2}}}  \left[  \frac{1}{1+ (\frac{d(x)}{2c})^{Q + \delta}}  \int_{d(y) > \frac{d(x)}{2c}}  \exp^{\frac{-d(x,y)^2}{C t}} dy \right. \\
&+& \left. \int_{d(y) \leq \frac{d(x)}{2c}}  \frac{1}{1+ d(y)^{Q + \delta}} \exp^{\frac{-d(x,y)^2}{C t}} dy \right] \\
& \leq & \frac{\epsilon C}{t^{\frac{Q}{2}}}  \left[  \frac{t^{\frac{Q}{2}}}{1+ (\frac{d(x)}{2c})^{Q + \delta}} + \int_{d(y) \leq \frac{d(x)}{2c}}  \frac{1}{1+ d(y)^{Q + \delta}} \exp^{\frac{-d(x,y)^2}{C t}} dy \right] \\
& = & \frac{\epsilon C}{t^{\frac{Q}{2}}}  \left[  \frac{t^{\frac{Q}{2}}}{1+ (\frac{d(x)}{2})^{Q + \delta}} \right. \\
&+& \left. \int_{d(y) \leq \frac{d(x)}{2c}}  \frac{d(x,y)^{Q}}{(1+ d(y)^{Q + \delta}) d(x,y)^{Q}} \exp^{\frac{-d(x,y)^2}{C t}} dy \right] \\
& \leq & \frac{\epsilon C}{t^{\frac{Q}{2}}}  \left[  \frac{t^{\frac{Q}{2}}}{1+ (\frac{d(x)}{2})^Q} \right. \\
&+& \left. \frac{(2c)^Q}{d(x)^Q} \int_{d(y) \leq \frac{d(x)}{2}}  \frac{d(x,y)^{Q}}{(1+ d(y)^{Q + \delta})} \exp^{\frac{-d(x,y)^2}{2 C t}} \exp^{\frac{-d(x,y)^2}{2 C t}} dy \right] \\
& \leq & \frac{\epsilon C}{1+ d(x)^Q} + \frac{\epsilon C}{ d(x)^Q} \leq \frac{\epsilon  2 C}{1+ d(x)^{Q-2}} + \frac{\epsilon C}{ d(x)^Q}\\
&\leq&   \frac{\epsilon  2 C}{1+ d(x)^{Q-2}},
\end{eqnarray*}
	where the third last inequality comes from H\"older's inequality since $\int_\G \exp^{\frac{-d(x,y)^2}{2 C t}} dy \leq C t^{\frac{Q}{2}}$ and $\frac{d(x,y)^{Q}}{(1+ d(y)^{Q + \delta})} \exp^{\frac{-d(x,y)^2}{2 C t}} \leq C$, and the last inequality is true whenever $d(x) > 1$.

To attend to the other case, let $$S = \left\{ x \in \G: d(x) \leq 1 \right\}.$$
Then $S$ is a compact and hence the continuous function $x  \to \int_{d(y) \leq \frac{d(x)}{2}}  \frac{1}{1+ d(y)^{Q + \delta}} \exp^{\frac{-d(x,y)^2}{C t}} dy$ is bounded by $t^{\frac{Q}{2}} C$ on $S$.
In this case,
\begin{eqnarray*}
e^{-t \Delta_\G} u(0)(x) &\leq& \frac{\epsilon C}{t^{\frac{Q}{2}}}  \left[  \frac{1}{1+ (\frac{d(x)}{2c})^{Q + \delta}}  \int_{d(y) > \frac{d(x)}{2c}}  \exp^{\frac{-d(x,y)^2}{C t}} dy \right. \\
&+& \left. \int_{d(y) \leq \frac{d(x)}{2c}}  \frac{1}{1+ d(y)^{Q + \delta}} \exp^{\frac{-d(x,y)^2}{C t}} dy \right] \\
& \leq & \frac{\epsilon C}{t^{\frac{Q}{2}}}  \left[  \frac{t^{\frac{Q}{2}}}{1+ (\frac{d(x)}{2c})^{Q + \delta}} + t^{\frac{Q}{2}} C \right] \\
& \leq & \frac{\epsilon C}{t^{\frac{Q}{2}}}  \left[  \frac{t^{\frac{Q}{2}}}{1+ (\frac{d(x)}{2c})^{Q}} + \frac{2 t^{\frac{Q}{2}} C}{1+ d(x)^Q} \right] (\text{as } d(x) \leq 1) \\
&\leq& \frac{\epsilon C}{1+ d(x)^Q} \leq   \frac{\epsilon  2 C}{1+ d(x)^{Q-2}}.
\end{eqnarray*} 
Hence, $e^{-t \Delta_\G} u(0)(x) \leq \frac{C}{1 + d(x)^{Q-2}}$ for every $x \in \G$ and $t \in [0, \infty)$.

Now that we have the sought after bound for $e^{-t \Delta_\G} u(0)$, let us work with $\left (\int_{0}^t e^{-(t-s) \Delta_\G} (u(s)^p + f) ds \right )(x)$.
From the choice of $\delta$, we obtain
\begin{eqnarray*}
 \left (\int_{0}^t e^{-(t-s) \Delta_\G} (u(s)^p + f) ds \right )(x)  &=& \int_0^t e^{-(t-s) \Delta_\G} (u(s)^p(x) + f(x)) ds \\
&=& \int_0^t \int_\G h_{(t-s)}(x,y) (u(s)^p(y) + f(y)) dy ds \\
&\leq&  \int_{0}^t \int_\G h_{(t-s)}(x,y) \left(\frac{M^p}{(1 + d(y)^{Q-2})^p} + \frac{\epsilon}{1+ d(y)^{Q + \delta}}\right) dy ds\\
&\leq & \int_{0}^t \int_\G h_{(t-s)}(x,y) \left(\frac{C M^p}{1 + d(y)^{Q+ \delta}} + \frac{\epsilon}{1+ d(y)^{Q + \delta}}\right) dy ds\\
&=& \int_{0}^t \int_\G  h_{(t-s)}(x,y)\frac{C M^p + \epsilon}{1 + d(y)^{Q+ \delta}} dy ds.
\end{eqnarray*}

If $H$ denotes the fundamental solution for $\Delta_\G$, then $\Delta_\G (H) = 0$ in $\G^{\ast}$.
It follows from the definition of the heat kernel $h_t(x,y)$ that for every fixed $y \in \G$, $\frac{\partial h_t(x,y)}{\partial t} - \Delta_\G h_t(x,y) = 0$.
Integrating this equation over time and using (\ref{BoundsForHeatKernel})  we obtain $$0=\lim_{s \to \infty} \int_0^s \left(\frac{\partial h_t(x,y)}{\partial t} - \Delta_\G h_t(x,y) \right) dt =  - \lim_{s \to \infty}  \Delta_\G \int_0^s  h_t(x,y) dt.$$
Non-negativity of the heat kernel now gives  $$\int_0^s h_{s-t}(x,y) dt \leq \lim_{s \to \infty} \int_0^s h_{s-t}(x,y) dt = H(x,y).$$
Since $H(x,y) = \frac{\beta_d}{d(x,y)^{Q-2}}$ for some $\beta_d > 0$, using \Cref{Proposition2OfZhang} we obtain
\begin{eqnarray*}
\left (\int_{0}^t e^{-(t-s) \Delta_\G} (u(s)^p + f) ds \right )(x)  &\leq& \int_{0}^t \int_\G  h_{(t-s)}(x,y)\frac{C M^p + \epsilon}{1 + d(y)^{Q+ \delta}} dy ds \\
& \leq & \int_\G  H(x,y)\frac{C M^p + \epsilon}{1 + d(y)^{Q+ \delta}} dy\\
& \leq & \int_\G \frac{\beta_d}{d(x,y)^{Q-2}} \frac{C M^p + \epsilon}{1 + d(y)^{Q+ \delta}} dy\\
&\leq&  \frac{C (C M^p + \epsilon)(\beta_d)}{1+ d(x)^{Q-2}}.
\end{eqnarray*}
Thus, there exists small $M$ and $\epsilon$ such that  $T_M$ is well-defined on $S_M$.
	
To prove that $T_M$ is a contraction, note that $(Q-2)(p-1) - 2 = \delta > 0$  and hence \begin{eqnarray*}
\norm{ T_M(u_1) - T_M(u_2) }_{\sup} &=& \sup_{t \in [0, \infty)} \left\| \int_{0}^t \exp^{-(t-s)\Delta_\G}( u_1(s)^p - u_2(s)^p ) \right\|\\
& \leq & p \norm{u_1 - u_2}_{\sup} \sup_{t \in [0, \infty)}  \int_\G \int_{0}^t h_{(t-s)}(x,y) \left(\frac{M}{1 + d(y)^{Q-2}}\right)^{p-1} dy \\
& \leq & p M^{p-1} \norm{u_1 - u_2}_{\sup}  \sup_{x \in \G}  \int_\G H(x,y) \frac{1}{1 + d(y)^{2 + \delta}} dy \\
&\leq & C p M^{p-1} \norm{u_1 - u_2}_{\sup}
\end{eqnarray*}
where $\int_\G H(x,y) \frac{1}{1 + d(x,0)^{2 + \delta}} dy < \infty$ can be seen by using $w = 2$ in \Cref{Proposition1OfZhang}.
This completes the proof.
\end{proof}

\end{document}